\documentclass[11pt]{amsart}

\usepackage{amsmath,amssymb,amsthm}
\usepackage{mathtools}
\usepackage{graphicx}
\usepackage{enumitem}
\usepackage{tikz}
\usepackage{hyperref}
\usepackage{comment}

\theoremstyle{plain}
\newtheorem{thm}{Theorem}[section]
\newtheorem{lem}[thm]{Lemma}
\newtheorem{prop}[thm]{Proposition}
\newtheorem{cor}[thm]{Corollary}
\newtheorem{conj}[thm]{Conjecture}

\theoremstyle{definition}
\newtheorem{defi}[thm]{Definition}

\theoremstyle{remark}
\newtheorem{rem}[thm]{Remark}

\newcommand{\Fb}{\mathbb{F}}
\newcommand{\Rb}{\mathbb{R}}
\newcommand{\Zb}{\mathbb{Z}}
\newcommand{\Ac}{\mathcal{A}}
\newcommand{\Hc}{\mathcal{H}}
\newcommand{\Pc}{\mathcal{P}}
\newcommand{\Wc}{\mathcal{W}}
\DeclareMathOperator{\Hom}{Hom}
\DeclareMathOperator{\Ext}{Ext}
\DeclareMathOperator{\Fix}{Fix}
\DeclareMathOperator{\rk}{rank}
\DeclareMathOperator{\id}{id}
\DeclareMathOperator{\ima}{im}
\newcommand{\Gal}{\operatorname{Gal}}
\newcommand{\grh}{\mathcal{G}}

\title{Discrete homotopy of graphs and $k$-parabolic subspace arrangements}

\author{So Yamagata}
\address{Department of Applied Mathematics, Fukuoka University}
\email{so.yamagata@fukuoka-u.ac.jp}

\date{}

\begin{document}

\begin{abstract}
  Discrete homotopy theory is a combinatorial homotopy theory defined for discrete objects.
  Barcelo--Severs--White conjectured that the homotopy groups of complements of $k$-parabolic subspace arrangements are isomorphic to the discrete homotopy groups of graphs associated with the Coxeter complexes.
  We show that this correspondence fails in general by giving a counterexample in degree two for the $4$-parabolic subspace arrangement of type $A_5$.
\end{abstract}

\maketitle

\section{Introduction}
Discrete homotopy theory, also known as $A$-homotopy theory, is a combinatorial homotopy theory for graphs and related discrete objects, including simplicial complexes, metric spaces, and posets.
Its origins can be traced back to the work of Atkin~\cite{Atk74,Atk76}, who introduced a combinatorial framework for studying connectivity in simplicial complexes, motivated in part by applications to the analysis of networks.
Building on Atkin's ideas, Barcelo, Kramer, Laubenbacher, and Weaver~\cite{BKLW01} developed a systematic connectivity theory for simplicial complexes and introduced the corresponding $A$-homotopy groups.
A central role in this theory is played by connectivity graphs associated with simplicial complexes, which encode incidence relations among their faces and allow connectivity properties of simplicial complexes to be studied through discrete homotopy invariants of graphs.
Further developments and applications of this viewpoint were discussed by Barcelo and Laubenbacher~\cite{BL05}.

A cubical approach to discrete homotopy theory was introduced in~\cite{BBdLL06}, where a cubical set is associated with a graph and its homotopy groups are compared with the $A$-homotopy groups of the graph.
This relationship has since been placed on a homotopy-theoretic foundation by Carranza and Kapulkin~\cite{CK23,CK24}.
In particular, the cubical nerve associated with a graph provides a topological model for its discrete homotopy groups, and this framework also gives a discrete Hurewicz theorem relating the first non-vanishing $A$-homotopy group to discrete cubical homology.
Related developments include a systematic study of the discrete fundamental groupoid and covering graphs~\cite{KM25}.
More recently, the comparison between discrete and classical homotopy theory has been extended to homotopy $n$-types~\cite{CK26}.

In parallel with these homotopy-theoretic developments, discrete cubical homology has become increasingly amenable to explicit computation.
Following the introduction of discrete homotopy and associated cubical homology in the setting of metric spaces~\cite{BCW14}, computational methods for discrete homology have developed substantially.
In particular, efficient algorithms for the first and higher discrete homology groups, as well as reduction techniques for higher-dimensional computations, have been developed~\cite{EK25,EK26,EKK26}.

One source of applications of discrete homotopy theory is the topology of subspace arrangements.
Barcelo and Smith~\cite{BS08} used discrete fundamental groups of connectivity graphs associated with Boolean lattices to obtain combinatorial information about the complements of $3$-equal subspace arrangements.
Subsequently, Barcelo, Severs, and White~\cite{BSW11} introduced $k$-parabolic subspace arrangements associated with finite real reflection groups and proposed a broader relationship between complements of subspace arrangements and discrete homotopy groups of Coxeter complexes.
More precisely, let $W$ be a finite real reflection group of rank $n$, let $C(W)$ be its Coxeter complex, and let $\Wc_{n,k}(W)$ denote its $k$-parabolic subspace arrangement.
They conjectured that the homotopy groups of the complement of $\Wc_{n,k}(W)$ are isomorphic to the discrete homotopy groups of the corresponding connectivity graph of $C(W)$.

\begin{conj}[{\cite[Conjecture~6.2]{BSW11}}]\label{conj:BSW}
  For every $m\geq0$, there is an isomorphism
  \begin{equation*}
    \pi_m\bigl(M(\Wc_{n,k}(W))\bigr)\cong A_m\bigl(\Gamma^{n-k+1}(C(W))\bigr).
  \end{equation*}
\end{conj}

In type $A$, the $k$-parabolic subspace arrangement is the essentialization of the classical $k$-equal subspace arrangement.
Conjecture~\ref{conj:BSW} therefore asserts a direct correspondence between the ordinary homotopy theory of $k$-equal subspace arrangements and the discrete homotopy theory of connectivity graphs of Coxeter complexes.

Our main result shows that this correspondence does not hold in general.
In fact, the failure already occurs at degree two.
For the Coxeter group of type $A_5$ and the $4$-parabolic arrangement, the conjectural correspondence in degree two compares a second discrete homotopy group with the second homotopy group of the arrangement complement, and we show that these groups are not isomorphic.
The groups remain nonisomorphic after tensoring with $\mathbb{F}_2$.

For the Coxeter group of type $A_n$, we write $G_{n,k}=\Gamma^{n-k+1}(C(A_n))$.

\begin{thm}\label{thm:main}
  We have
  \begin{equation*}
    \pi_2\bigl(M(\Wc_{5,4}(A))\bigr)\otimes\Fb_2\cong\Fb_2^{49},\qquad A_2(G_{5,4})\otimes\Fb_2=0.
  \end{equation*}
\end{thm}
The ordinary homotopy group in this example follows from the known topology of $k$-equal subspace arrangements.
The results of Bj\"{o}rner and Welker~\cite{BW95}, together with the Hurewicz theorem, imply that
\begin{equation*}
  \pi_2\bigl(M(\Wc_{5,4}(A))\bigr)\otimes\Fb_2\cong\Fb_2^{49}.
\end{equation*}
It remains to compute the discrete homotopy group.
Barcelo--Severs--White showed that $G_{5,4}$ is connected and $A_1(G_{5,4})=0$, so the discrete Hurewicz theorem~~\cite{CK24} reduces the problem to the computation of the second discrete cubical homology of $G_{5,4}$.

Our main tool for this computation is a filtration of the normalized cubical chain complex by parabolic support.
The filtration decomposes the homological calculation into local pieces determined by standard parabolic subgroups of $S_6$.
While recent algorithms provide increasingly efficient methods for computing discrete homology~\cite{EK25,EK26,EKK26}, our purpose here is to organize the computation in a way that uses the structure of $G_{5,4}$ explicitly.
Using this filtration, we prove that $DH_2(G_{5,4};\Fb_2)=0$, and hence
\begin{equation*}
  A_2(G_{5,4})\otimes\Fb_2=0.
\end{equation*}

The paper is organized as follows.
In Section~\ref{sec:discrete_htp}, we recall the basic notions of discrete homotopy and discrete cubical homology and their cubical models.
In Section~\ref{sec:k-parabolic}, we recall $k$-parabolic subspace arrangements and establish the ordinary homotopy part of Theorem~\ref{thm:main}.
In Section~\ref{sec:discrete-side}, we study the graph $G_{5,4}$ and prove the required vanishing result using the support-rank filtration.
The homology calculation is computer-assisted, and the computational details and verification procedures are given in Appendix~\ref{app:computation}.

\section{Discrete homotopy and homology of graphs}\label{sec:discrete_htp}
In this section, we recall the basic definitions and results in discrete homotopy theory for graphs and the associated cubical theory, following~\cite{BL05,BBdLL06,CK24}.

\subsection{Connectivity graphs and discrete homotopy theory}
We begin by recalling how a simplicial complex gives rise to a graph through its connectivity graph.

\begin{defi}
  A \textbf{graph} $G$ consists of a vertex set $V_G$ together with a reflexive and symmetric relation $\sim$ on $V_G$.
\end{defi}

\begin{defi}[{\cite{BKLW01}}]
  Let $\Delta$ be a simplicial complex of dimension $d$ and let $0\leq q\leq d$.
  The \textbf{$q$-connectivity graph $\Gamma^q(\Delta)$} of $\Delta$ is the graph whose vertices are the maximal simplices of $\Delta$ of dimension at least $q$, with $\sigma\sim\tau$ if and only if $\dim(\sigma\cap\tau)\geq q$.
\end{defi}
Figure~\ref{fig:conn_grph} illustrates an example of the connectivity graphs associated with a simplicial complex.
\begin{figure}[htbp]
  \centering
  \begin{tabular}{cccc}
    \begin{minipage}{0.22\linewidth}
      \centering
      \begin{tikzpicture}[scale=0.5]
        \foreach \angle/\nextangle in {90/162,162/234,234/306,306/18,18/90}
        \filldraw[fill=gray!50] (0,0) -- (\angle:2cm) -- (\nextangle:2cm) -- cycle;
        \foreach \angle in {90,162,234,306,18,90}
        \draw (0,0) -- (\angle:2cm);
        \foreach \angle in {90,162,234,306,18,90}
        \node[draw,circle,fill=black,inner sep=2pt] at (\angle:2cm) {};
        \node[draw,circle,fill=black,inner sep=2pt] at (0,0) {};
      \end{tikzpicture}
    \end{minipage}
     &
    \begin{minipage}{0.22\linewidth}
      \centering
      \begin{tikzpicture}[scale=0.4]
        \foreach \angle/\label in {198/a,270/b,342/c,54/d,126/e}
        \node[draw,circle,fill=black,inner sep=2pt] (\label) at (\angle:2cm) {};
        \draw (a) -- (b); \draw (a) -- (c); \draw (a) -- (d);
        \draw (b) -- (c); \draw (b) -- (d); \draw (b) -- (e);
        \draw (c) -- (d); \draw (c) -- (e); \draw (d) -- (e);
        \draw (e) -- (a);
      \end{tikzpicture}
    \end{minipage}
     &
    \begin{minipage}{0.22\linewidth}
      \centering
      \begin{tikzpicture}[scale=0.4]
        \foreach \angle/\label in {198/a,270/b,342/c,54/d,126/e}
        \node[draw,circle,fill=black,inner sep=2pt] (\label) at (\angle:2cm) {};
        \draw (a) -- (b); \draw (b) -- (c); \draw (c) -- (d);
        \draw (d) -- (e); \draw (e) -- (a);
      \end{tikzpicture}
    \end{minipage}
     &
    \begin{minipage}{0.22\linewidth}
      \centering
      \begin{tikzpicture}[scale=0.4]
        \foreach \angle/\label in {198/a,270/b,342/c,54/d,126/e}
        \node[draw,circle,fill=black,inner sep=2pt] (\label) at (\angle:2cm) {};
      \end{tikzpicture}
    \end{minipage}
  \end{tabular}
  \caption{The leftmost figure shows a simplicial complex $\Delta$, and the other figures show the associated connectivity graphs $\Gamma^0(\Delta), \Gamma^1(\Delta), \Gamma^2(\Delta)$.}\label{fig:conn_grph}
\end{figure}
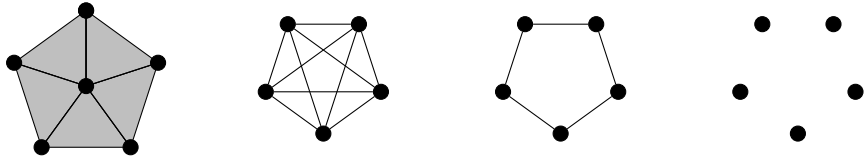

We next recall the graph-theoretic notions used in the definition of discrete homotopy groups.

\begin{defi}
  For graphs $G$ and $H$, we make the following definitions.
  \begin{itemize}
    \item the \textbf{Cartesian product} $G\Box H$ is the graph with vertex set $V_G\times V_H$ such that
          \begin{equation*}
            (u_1,v_1)\sim(u_2,v_2)
          \end{equation*}
          if either $u_1=u_2$ and $v_1\sim v_2$, or $u_1\sim u_2$ and $v_1=v_2$;
    \item a \textbf{graph map} $f:G\to H$ is a set map $V_G\to V_H$ such that
          \begin{equation*}
            u\sim v\Longrightarrow f(u)\sim f(v).
          \end{equation*}
  \end{itemize}
\end{defi}

\begin{defi}
  Let $G$ be a graph.
  A \textbf{subgraph} $A$ of $G$ is a graph such that $V_A\subseteq V_G$ and, if $u\sim_A v$, then $u\sim_G v$.
  Let $A\subseteq G$ and $B\subseteq H$ be subgraphs.
  A \textbf{relative graph map} $f:(G,A)\longrightarrow(H,B)$ is a graph map $f:G\to H$ such that $f(A)\subseteq B$.
\end{defi}

\begin{defi}
  For an integer $m\geq0$, let $I_m$ be the graph with vertex set $V_{I_m}=\{0,1,\dots,m\}$, where two vertices $i,j$ are adjacent if and only if $|i-j|\leq1$.
  We also use the infinite interval graph $I_\infty$, whose vertex set is $\Zb$, with the same adjacency condition.

  For $n\geq1$ and $M\geq0$, let $I_{\geq M}^{\Box n}$ be the induced subgraph of $I_\infty^{\Box n}$ on the vertex set
  \begin{equation*}
    \{(u_1,\dots,u_n)\in\Zb^n\mid |u_i|\geq M\text{ for some }i\}.
  \end{equation*}
  See Figures~\ref{fig:I_infty} and~\ref{fig:I_geqM} for examples.
\end{defi}

\begin{figure}[h]
  \centering
  \begin{minipage}[b]{0.45\linewidth}
    \centering
    \raisebox{0.3cm}{
      \begin{tikzpicture}
        \foreach \x in {-1.5,-1,-0.5,0,0.5,1,1.5}{
            \foreach \y in {-1.5,-1,-0.5,0,0.5,1,1.5}{
                \fill (\x,\y) circle (1.5pt);
              }
            \draw (\x,-1.75) -- (\x,1.75);
            \node at (\x,-1.9) {$\vdots$};
            \node at (\x,2.1) {$\vdots$};
          }
        \foreach \y in {-1.5,-1,-0.5,0,0.5,1,1.5}{
            \draw (-1.75,\y) -- (1.75,\y);
            \node at (-2,\y) {$\dots$};
            \node at (2.1,\y) {$\dots$};
          }
      \end{tikzpicture}
    }
    \caption{The graph $I_{\infty}^{\Box 2}$.}
    \label{fig:I_infty}
  \end{minipage}
  \hfill
  \begin{minipage}[b]{0.45\linewidth}
    \centering
    \begin{tikzpicture}
      \foreach \x in {-1.5,-1,1,1.5}{
          \foreach \y in {-1.5,-1,-0.5,0,0.5,1,1.5}{
              \fill (\x,\y) circle (1.5pt);
            }
          \draw (\x,-1.75) -- (\x,1.75);
        }
      \foreach \x in {-0.5,0,0.5}{
          \foreach \y in {-1.5,-1,1,1.5}{
              \fill (\x,\y) circle (1.5pt);
            }
          \draw (\x,-1.75) -- (\x,-1);
          \draw (\x,1) -- (\x,1.75);
        }
      \foreach \y in {-1.5,-1,1,1.5}{
          \draw (-1.75,\y) -- (1.75,\y);
        }
      \foreach \y in {-0.5,0,0.5}{
          \draw (-1.75,\y) -- (-1,\y);
          \draw (1,\y) -- (1.75,\y);
        }
      \foreach \x in {-1.5,-0.5,0.5,1.5}{
          \node at (\x,-1.9) {$\vdots$};
        }
      \foreach \x in {-1.5,-1,-0.5,0,0.5,1,1.5}{
          \node at (\x,2.1) {$\vdots$};
        }
      \foreach \y in {-1.5,-1,-0.5,0,0.5,1,1.5}{
          \node at (-2,\y) {$\dots$};
          \node at (2.1,\y) {$\dots$};
        }
      \node at (2.35,0) {$0$};
      \node at (2.35,1) {$M$};
      \node at (2.35,-1) {$-M$};
      \node at (0,-2) {$0$};
      \node at (1,-2) {$M$};
      \node at (-1,-2) {$-M$};
    \end{tikzpicture}
    \caption{The graph $I_{\geq M}^{\Box 2}$.}
    \label{fig:I_geqM}
  \end{minipage}
\end{figure}

We now define discrete homotopy of graph maps.

\begin{defi}[{\cite{BL05}}]
  Two graph maps $f,g:G\to H$ are called \textbf{discrete homotopic}, denoted by $f\simeq g$, if there exist an integer $m\geq0$ and a graph map $F:G\Box I_m\longrightarrow H$ such that $F(\bullet,0)=f$ and $F(\bullet,m)=g$.
  We call $F$ a \textbf{discrete homotopy} from $f$ to $g$.
\end{defi}

The notion of discrete homotopy is different from ordinary homotopy of the underlying one-dimensional simplicial complex.
For example, let $C_n$ denote the cycle graph with $n$ vertices.
Every $C_n$ is homotopy equivalent to $S^1$ as a one-dimensional simplicial complex.
However, in discrete homotopy theory, the small cycles $C_3$ and $C_4$ are discrete homotopy equivalent to a single vertex.
Indeed, for $n=3,4$, there exist discrete homotopies $F:C_n\Box I_m\to C_n$ from $\id_{C_n}$ to the constant graph map $c_0$ (see Figures~\ref{fig:C3} and~\ref{fig:C4}).
On the other hand, $C_n$ is not discrete homotopy equivalent to a single vertex for $n\geq5$.

\begin{figure}[h]
  \begin{minipage}[b]{0.48\columnwidth}
    \centering
    \begin{tikzpicture}
      \fill (0,0) circle (1.5pt);
      \fill (1.6,0) circle (1.5pt);
      \fill (0.8,0.8) circle (1.5pt);
      \fill (0,1.5) circle (1.5pt);
      \fill (1.6,1.5) circle (1.5pt);
      \fill (0.8,2.3) circle (1.5pt);
      \fill (3.5,0.5) circle (1.5pt);
      \fill (5.1,0.5) circle (1.5pt);
      \fill (4.3,1.8) circle (1.5pt);
      \draw (0,0) -- (1.6,0) -- (0.8,0.8) --cycle;
      \draw (0,1.5) -- (1.6,1.5) -- (0.8,2.3) -- cycle;
      \draw (0,0) -- (0,1.5) -- (1.6,1.5) -- (1.6,0) -- cycle;
      \draw (0.8,0.8) -- (0.8,1.4);
      \draw (0.8,1.6) -- (0.8,2.3);
      \draw[->] (2,1) -- (3,1);
      \draw (3.5,0.5) -- (5.1,0.5) -- (4.3,1.8) -- cycle;
      \draw[->] (1.6, 2) .. controls (2.6,2.6) .. (3.5, 1.7);
      \draw[->] (1.6, -0.5) .. controls (2.6,-1.1) .. (3.5, -0.2);
      \node at (-0.2, 0) {0};
      \node at (1.8, 0) {1};
      \node at (1, 0.9) {2};
      \node at (-0.2, 1.5) {0};
      \node at (1.8, 1.5) {0};
      \node at (0.8, 2.5) {0};
      \node at (3.5, 0.2) {0};
      \node at (5.1, 0.2) {1};
      \node at (4.3, 2.1) {2};
      \node at (2.6,2.6) {$c_0$};
      \node at (2.6,-1.2) {$\id_{C_3}$};
      \node at (2.5, 1.3) {$F$};
    \end{tikzpicture}\caption{Discrete homotopy $C_3 \Box I_1 \to C_3$ between $c_0$ and $\id_{C_3}$.}\label{fig:C3}
  \end{minipage} \quad
  \begin{minipage}[b]{0.48\columnwidth}
    \centering
    \begin{tikzpicture}
      \fill (0,0) circle (1.5pt);
      \fill (1.6,0) circle (1.5pt);
      \fill (2,1) circle (1.5pt);
      \fill (0.4,1) circle (1.5pt);
      \fill (0,1.5) circle (1.5pt);
      \fill (1.6,1.5) circle (1.5pt);
      \fill (2,2.5) circle (1.5pt);
      \fill (0.4,2.5) circle (1.5pt);
      \fill (0,3) circle (1.5pt);
      \fill (1.6,3) circle (1.5pt);
      \fill (2,4) circle (1.5pt);
      \fill (0.4,4) circle (1.5pt);
      \fill (4,1.2) circle (1.5pt);
      \fill (5.6,1.2) circle (1.5pt);
      \fill (5.6,2.8) circle (1.5pt);
      \fill (4,2.8) circle (1.5pt);
      \draw (0.4,1) -- (0,0) -- (1.6,0) -- (2,1);
      \draw (0.4,2.5) -- (0,1.5) -- (1.6,1.5) -- (2,2.5);
      \draw (0,3) -- (1.6,3) -- (2,4) -- (0.4,4) -- cycle;
      \draw (0,0) -- (0,1.5) -- (0,3);
      \draw (1.6,0) -- (1.6,1.5) -- (1.6,3);
      \draw (2,1) -- (2,2.5) -- (2,4);
      \draw (0.4,1) -- (1.5,1);
      \draw (1.7,1) -- (2,1);
      \draw (0.4,2.5) -- (1.5,2.5);
      \draw (1.7,2.5) -- (2,2.5);
      \draw (0.4,1) -- (0.4,1.4);
      \draw (0.4,1.6) -- (0.4,2.5);
      \draw (0.4,2.5) -- (0.4,2.9);
      \draw (0.4,3.1) -- (0.4,4);
      \draw[->] (2.5,2) -- (3.5,2);
      \draw (4,1.2) -- (5.6,1.2) -- (5.6,2.8) -- (4,2.8) -- cycle;
      \draw[->] (2.5, 4) .. controls (3.4,4) .. (4, 3.5);
      \draw[->] (2.5, 0) .. controls (3.4,0) .. (4, 0.5);
      \node at (-0.2, 0) {0};
      \node at (1.8,0) {1};
      \node at (2.2,1) {2};
      \node at (0.2,1) {3};
      \node at (-0.2,1.5) {0};
      \node at (1.8,1.5) {0};
      \node at (2.2,2.5) {3};
      \node at (0.2,2.5) {3};
      \node at (-0.2,3) {0};
      \node at (1.8,3) {0};
      \node at (2.2,4) {0};
      \node at (0.2,4) {0};
      \node at (3,2.3) {$F$};
      \node at (4,0.9) {0};
      \node at (5.6,0.9) {1};
      \node at (5.6,3.1) {2};
      \node at (4,3.1) {3};
      \node at (3.4,4.1) {$c_0$};
      \node at (3.5,-0.3) {$\id_{C_4}$};
    \end{tikzpicture}\caption{Discrete homotopy $C_4 \Box I_2 \to C_4$ between $c_0$ and $\id_{C_4}$.}\label{fig:C4}
  \end{minipage}
\end{figure}

\begin{defi}
  Two relative graph maps $f,g:(G,A)\longrightarrow(H,B)$ are said to be \textbf{relatively discrete homotopic}, denoted by $f\simeq g$, if there exist an integer $m\geq0$ and a graph map $F:G\Box I_m\longrightarrow H$ such that $F(\bullet,0)=f$, $F(\bullet,m)=g$, and $F|_{A\Box I_m}:A\Box I_m\longrightarrow B$ is a discrete homotopy from $f|_A$ to $g|_A$.
  We call $F$ a \textbf{relative discrete homotopy} from $f$ to $g$.
\end{defi}

\begin{defi}[{\cite{CK24}}]
  Let $n\geq1$ and $u_0\in V_G$.
  The \textbf{$n$-th discrete homotopy group of $G$ at $u_0$} is the set
  \begin{equation*}
    A_n(G,u_0)=\left\{(I_\infty^{\Box n},I_{\geq M}^{\Box n})\longrightarrow(G,u_0)\ \middle|\ M\geq0\right\}/\simeq,
  \end{equation*}
  where $\simeq$ denotes relative discrete homotopy.
\end{defi}

\begin{rem}
  The set $A_0(G,u_0)$ is defined to be the pointed set of connected components of $G$, with distinguished element the component containing $u_0$.
  For $n\geq1$, $A_n(G,u_0)$ has a group structure, and it is abelian for $n\geq2$ (see, for instance,~\cite[Proposition~3.5]{BBdLL06}).
\end{rem}

\subsection{Cubical model for discrete homotopy theory}
We review the cubical model for discrete homotopy theory, following~\cite{CK23,CK24}.
We use cubical sets with positive and negative connections.

Let $\square$ denote the cubical category whose objects are
\begin{equation*}
  [1]^n=\{0\leq1\}^n,\qquad n\geq0.
\end{equation*}
Its morphisms are generated under composition by the following maps:
\begin{itemize}
  \item the \textbf{face maps} $\partial_{i,\varepsilon}^n:[1]^{n-1}\to[1]^n$, for $1\leq i\leq n$ and $\varepsilon\in\{0,1\}$, given by
        \begin{equation*}
          \partial_{i,\varepsilon}^n(x_1,\dots,x_{n-1})=(x_1,\dots,x_{i-1},\varepsilon,x_i,\dots,x_{n-1});
        \end{equation*}
  \item the \textbf{degeneracy maps} $\sigma_i^n:[1]^n\to[1]^{n-1}$, for $1\leq i\leq n$, given by
        \begin{equation*}
          \sigma_i^n(x_1,\dots,x_n)=(x_1,\dots,x_{i-1},x_{i+1},\dots,x_n);
        \end{equation*}
  \item the \textbf{negative connection maps} $\gamma_{i,0}^n:[1]^n\to[1]^{n-1}$, for $1\leq i\leq n-1$, given by
        \begin{equation*}
          \gamma_{i,0}^n(x_1,\dots,x_n)=(x_1,\dots,x_{i-1},\max\{x_i,x_{i+1}\},x_{i+2},\dots,x_n);
        \end{equation*}
  \item the \textbf{positive connection maps} $\gamma_{i,1}^n:[1]^n\to[1]^{n-1}$, for $1\leq i\leq n-1$, given by
        \begin{equation*}
          \gamma_{i,1}^n(x_1,\dots,x_n)=(x_1,\dots,x_{i-1},\min\{x_i,x_{i+1}\},x_{i+2},\dots,x_n).
        \end{equation*}
\end{itemize}

These maps satisfy the following cubical identities, where superscripts are omitted:
\begin{align*}
  \partial_{j,\varepsilon'}\partial_{i,\varepsilon}
   & =
  \partial_{i+1,\varepsilon}\partial_{j,\varepsilon'} \quad \text{for }j\leq i,
   &   &                                                                                     \\
  \sigma_j\partial_{i,\varepsilon}
   & =
  \begin{cases}
    \partial_{i-1,\varepsilon}\sigma_j,   & j<i, \\
    \id,                                  & j=i, \\
    \partial_{i,\varepsilon}\sigma_{j-1}, & j>i,
  \end{cases}                                               \\
  \sigma_i\sigma_j
   & =
  \sigma_j\sigma_{i+1} \quad \text{for }j\leq i,
   &   &                                                                                     \\
  \gamma_{j,\varepsilon'}\gamma_{i,\varepsilon}
   & =
  \begin{cases}
    \gamma_{i,\varepsilon}\gamma_{j+1,\varepsilon'}, & j>i,                            \\
    \gamma_{i,\varepsilon}\gamma_{i+1,\varepsilon},  & j=i,\ \varepsilon'=\varepsilon,
  \end{cases}         \\
  \gamma_{j,\varepsilon'}\partial_{i,\varepsilon}
   & =
  \begin{cases}
    \partial_{i-1,\varepsilon}\gamma_{j,\varepsilon'}, & j<i-1,                                \\
    \id,                                               & j=i-1,i,\ \varepsilon=\varepsilon',   \\
    \partial_{j,\varepsilon}\sigma_j,                  & j=i-1,i,\ \varepsilon=1-\varepsilon', \\
    \partial_{i,\varepsilon}\gamma_{j-1,\varepsilon'}, & j>i,
  \end{cases} \\
  \sigma_j\gamma_{i,\varepsilon}
   & =
  \begin{cases}
    \gamma_{i-1,\varepsilon}\sigma_j,   & j<i, \\
    \sigma_i\sigma_i,                   & j=i, \\
    \gamma_{i,\varepsilon}\sigma_{j+1}, & j>i.
  \end{cases}
\end{align*}

A \textbf{cubical set} is a presheaf $X:\square^{\operatorname{op}}\longrightarrow\mathbf{Sets}$, and a \textbf{cubical map} is a natural transformation of such presheaves.
We write $\mathbf{cSet}$ for the category of cubical sets and cubical maps.
For a cubical set $X$, we write $X_n=X([1]^n)$ and call its elements the \textbf{$n$-cubes} of $X$.
We write cubical operators on the right; for example, if $x\in X_n$, then $x\partial_{i,\varepsilon}$ denotes its $(i,\varepsilon)$-face.
A cube is called \textbf{degenerate} if it lies in the image of a degeneracy or connection map.

For $m\geq1$, we regard $I_m^{\Box n}$ as the graph-theoretic model of the $n$-dimensional cube of length $m$.
The corresponding face, degeneracy, and connection maps are defined as follows:
\begin{itemize}
  \item the face map $\partial_{i,\varepsilon}^n:I_m^{\Box(n-1)}\to I_m^{\Box n}$, for $1\leq i\leq n$ and $\varepsilon\in\{0,1\}$, is given by
        \begin{equation*}
          \partial_{i,\varepsilon}^n(v_1,\dots,v_{n-1})=(v_1,\dots,v_{i-1},\varepsilon m,v_i,\dots,v_{n-1});
        \end{equation*}
  \item the degeneracy map $\sigma_i^n:I_m^{\Box n}\to I_m^{\Box(n-1)}$, for $1\leq i\leq n$, is given by
        \begin{equation*}
          \sigma_i^n(v_1,\dots,v_n)=(v_1,\dots,v_{i-1},v_{i+1},\dots,v_n);
        \end{equation*}
  \item the negative connection map $\gamma_{i,0}^n:I_m^{\Box n}\to I_m^{\Box(n-1)}$, for $1\leq i\leq n-1$, is given by
        \begin{equation*}
          \gamma_{i,0}^n(v_1,\dots,v_n)=(v_1,\dots,v_{i-1},\max\{v_i,v_{i+1}\},v_{i+2},\dots,v_n);
        \end{equation*}
  \item the positive connection map $\gamma_{i,1}^n:I_m^{\Box n}\to I_m^{\Box(n-1)}$, for $1\leq i\leq n-1$, is given by
        \begin{equation*}
          \gamma_{i,1}^n(v_1,\dots,v_n)=(v_1,\dots,v_{i-1},\min\{v_i,v_{i+1}\},v_{i+2},\dots,v_n).
        \end{equation*}
\end{itemize}

Let $\mathcal G$ denote the category of graphs and graph maps.
The maps above satisfy the cubical identities and therefore define a functor
\begin{equation*}
  Q_m:\square\longrightarrow\grh,\qquad [1]^n\longmapsto I_m^{\Box n}.
\end{equation*}
Consequently, for every graph $G$, the sets $\grh(I_m^{\Box n},G)$, together with the maps induced by precomposition, form a cubical set.

\begin{defi}[{\cite{CK24}}]\label{def:m-nerve}
  Let $m\geq1$ and let $G$ be a graph.
  The \textbf{$m$-nerve} of $G$ is the cubical set $N_mG$ defined by
  \begin{equation*}
    (N_mG)_n=\grh(I_m^{\Box n},G).
  \end{equation*}
  Its face, degeneracy, and connection maps are induced by precomposition with the corresponding maps between the cubical graphs $I_m^{\Box n}$.
\end{defi}

\begin{rem}
  The $1$-nerve of a graph $G$ coincides with the cubical set $M_*(G)$ introduced in~\cite{BBdLL06} (see also~\cite[Remark~3.2]{CK24}).
  Its geometric realization is the cubical complex $X_G$ associated with $G$.
\end{rem}

The following theorem shows that the cubical set $N_1G$ provides a topological model for the discrete homotopy groups of $G$.

\begin{thm}[{\cite[Theorem~5.1]{CK24}}]
  For every pointed graph $(G,v)$ and every $n\geq1$, there is a natural group isomorphism
  \begin{equation*}
    A_n(G,v)\cong\pi_n(|N_1G|,v).
  \end{equation*}
\end{thm}

\subsection{Discrete homology and the Hurewicz theorem}
Finally, we review the discrete homology groups of graphs and the corresponding Hurewicz theorem.
Since the cubical model of a graph is given by the cubical set $N_1G$, the relevant homology theory is its normalized cubical homology.

Let $G$ be a graph.
For $n\geq0$, let $L_n^D(G)$ be the free abelian group generated by graph maps $\sigma:I_1^{\Box n}\to G$.
Equivalently, the generators of $L_n^D(G)$ are the $n$-cubes of the cubical set $N_1G$.
For such a cube $\sigma$, $1\leq i\leq n$, and $\varepsilon\in\{0,1\}$, write $\sigma\partial_{i,\varepsilon}:I_1^{\Box(n-1)}\to G$ for its $(i,\varepsilon)$-face.
The boundary $\partial:L_n^D(G)\to L_{n-1}^D(G)$ is defined by
\begin{equation*}
  \partial\sigma=\sum_{i=1}^n\sum_{\varepsilon=0}^1(-1)^{i+\varepsilon}\sigma\partial_{i,\varepsilon}.
\end{equation*}
Let $D_n^D(G)\subseteq L_n^D(G)$ be the subgroup generated by degenerate cubes.
Here a cube is called degenerate if it lies in the image of a degeneracy map or a connection map of the cubical set $N_1G$.
Set
\begin{equation*}
  C_n^D(G)=L_n^D(G)/D_n^D(G).
\end{equation*}
The cubical identities imply that the boundary descends to the quotient, so $C_*^D(G)$ is the normalized cubical chain complex of $N_1G$.
The \textbf{discrete homology groups} of $G$ are
\begin{equation*}
  DH_n(G)=H_n(C_*^D(G)).
\end{equation*}
For a coefficient ring $R$, we write
\begin{equation*}
  DH_n(G;R)=H_n(C_*^D(G)\otimes R).
\end{equation*}
The construction is functorial in $G$.
Indeed, a graph map $f:G\to H$ sends a cube $\sigma:I_1^{\Box n}\to G$ to the composite $f\circ\sigma:I_1^{\Box n}\to H$, and hence induces a chain map $C_*^D(G)\to C_*^D(H)$.

For a pointed graph $(G,v)$, let $\widetilde{DH}_n(G,v)$ denote the reduced discrete homology group, namely the reduced normalized cubical homology of the pointed cubical set $(N_1G,v)$.
Similarly, for a coefficient ring $R$, we write $\widetilde{DH}_n(G,v;R)$ for the corresponding reduced homology group with coefficients in $R$.

The following theorem is the discrete version of the Hurewicz theorem for graphs.

\begin{thm}[{\cite[Theorem~5.8]{CK24}}]\label{thm:discrete-Hurewicz}
  Let $n\geq2$ and let $(G,v)$ be a pointed connected graph.
  Suppose that $A_i(G,v)=0$ for all $1\leq i<n$.
  Then the discrete Hurewicz homomorphism
  \begin{equation*}
    A_n(G,v)\longrightarrow\widetilde{DH}_n(G,v)
  \end{equation*}
  is an isomorphism.
\end{thm}

\section{$k$-parabolic subspace arrangements}\label{sec:k-parabolic}
\subsection{$k$-parabolic subspace arrangements and the type $A$ case}
We briefly recall the definition of the $k$-parabolic subspace arrangement introduced by Barcelo--Severs--White~\cite{BSW11}.
For background on hyperplane and subspace arrangements, see~\cite{Bjo94,OT92,Sta07}; for reflection groups, see~\cite{Hum90}.

\begin{defi}
  A \textbf{subspace arrangement} in a finite-dimensional real vector space $V$ is a finite collection $\Ac$ of linear subspaces of $V$ such that no element of $\Ac$ properly contains another.
  The \textbf{complement} of $\Ac$ is denoted by $M(\Ac)=V\setminus\bigcup_{X\in\Ac}X$.
\end{defi}

Recall that the \textbf{intersection poset} $L(\Ac)$ of $\Ac$ is the set of all intersections of elements of $\Ac$, ordered by reverse inclusion: $X\leq Y$ if and only if $X\supseteq Y$.
Let $W$ be a finite real reflection group of rank $n$ acting on a Euclidean vector space $V$.
Fix a set $S$ of simple reflections.
For a reflection $s\in W$, let $H_s=\{x\in V\mid sx=x\}$ be its reflecting hyperplane.
The collection of all reflecting hyperplanes of $W$ is called the \textbf{Coxeter arrangement} $\Hc(W)=\{H_s\mid s\in W\text{ is a reflection}\}$.

We recall the standard notions of parabolic subgroups of Coxeter groups (see~\cite{Hum90}).
\begin{defi}
  For $I\subseteq S$, set $W_I=\langle I\rangle$.
  The subgroup $W_I$ is called a \textbf{standard parabolic subgroup} of $W$, and a subgroup $G\subseteq W$ is called a \textbf{parabolic subgroup} of $W$ if $G=wW_Iw^{-1}$ for some $w\in W$ and $I\subseteq S$.
  The rank of $G$ is $|I|$.
  We call $G=wW_Iw^{-1}$ \textbf{irreducible} if the Coxeter diagram induced by $I$ is connected.
\end{defi}

For a subgroup $G\subseteq W$, set $\Fix(G)=\left\{x\in V\ \middle|\ gx=x\text{ for all }g\in G\right\}$.
Conversely, for $X\in L(\Hc(W))$, set $\Gal(X)=\left\{w\in W\ \middle|\ wx=x\text{ for all }x\in X\right\}$.
By the Galois correspondence due to Barcelo--Ihrig~\cite{BI99}, the assignments $G\longmapsto\Fix(G)$ and $X\longmapsto\Gal(X)$ give a correspondence between the poset of parabolic subgroups of $W$ and the intersection poset $L(\Hc(W))$ (see also~\cite{BSW11}).

\begin{defi}
  Let $\Pc_{n,k}(W)$ be the collection of irreducible parabolic subgroups of $W$ of rank $k-1$.
  The \textbf{$k$-parabolic subspace arrangement} associated with $W$ is
  \begin{equation*}
    \Wc_{n,k}(W)=\left\{\Fix(G)\ \middle|\ G\in\Pc_{n,k}(W)\right\}.
  \end{equation*}
\end{defi}

We now specialize to type $A$.
Let $W$ be of type $A_n$, so that $W=S_{n+1}$ acts on
\begin{equation*}
  V=\left\{(x_1,\ldots,x_{n+1})\in\Rb^{n+1}\ \middle|\ \sum_{i=1}^{n+1}x_i=0\right\}
\end{equation*}
by permuting coordinates.
The simple reflections are the adjacent transpositions $s_i=(i,i+1)$ for $1\leq i\leq n$.
For a $k$-element subset $T=\{i_1,\ldots,i_k\}\subseteq[n+1]\coloneqq\{1,\ldots,n+1\}$, define $X_T=\{x\in V\mid x_{i_1}=\cdots=x_{i_k}\}$.
Under the Galois correspondence, the irreducible parabolic subgroups of rank $k-1$ correspond precisely to the subspaces $X_T$.
Hence
\begin{equation*}
  \Wc_{n,k}(A)=\left\{X_T\ \middle|\ T\in\binom{[n+1]}{k}\right\}.
\end{equation*}
The $k$-equal subspace arrangement in $\Rb^{n+1}$ consists of the subspaces defined by $x_{i_1}=x_{i_2}=\cdots=x_{i_k}, 1\leq i_1<\cdots<i_k\leq n+1$.
We denote its complement by $M_{n+1,k}$.
Each of these subspaces contains the diagonal $D=\{(t,\ldots,t)\in\Rb^{n+1}\mid t\in\Rb\}$.
Since $\Rb^{n+1}=V\oplus D$, the essentialization of this arrangement is obtained by restricting it to $V$.
Consequently, the resulting arrangement is precisely $\mathcal W_{n,k}(A)$, and its complement satisfies $M_{n+1,k}\cong M(\mathcal W_{n,k}(A))\times\Rb$.
The topology and combinatorics of $k$-equal subspace arrangements have been studied extensively (see, for instance,~\cite{BW95,Kho96,SW12,BKK23}).

\subsection{The ordinary homotopy calculation}
We now establish the ordinary homotopy part of Theorem~\ref{thm:main}.
The calculation follows from the work of Bj\"{o}rner and Welker on the topology of $k$-equal subspace arrangements.

\begin{thm}[{\cite[Theorem~1.1]{BW95}}]\label{thm:BW}
  Let $M_{N,k}$ denote the complement of the $k$-equal subspace arrangement on $N$ coordinates.
  Then the following hold:
  \begin{itemize}
    \item $M_{N,k}$ is $(k-3)$-connected;
    \item the integral cohomology groups $H^d(M_{N,k};\Zb)$ are free abelian for all $d\geq0$;
    \item $H^d(M_{N,k};\Zb)\neq0$ if and only if $d=t(k-2)$ for some integer $t$ satisfying $0\leq t\leq\lfloor N/k\rfloor$;
    \item if $k\geq3$, then
          \begin{equation*}
            \rk H^{k-2}(M_{N,k};\Zb)=\sum_{i=k}^{N}\binom{N}{i}\binom{i-1}{k-1}.
          \end{equation*}
  \end{itemize}
\end{thm}

\begin{prop}\label{prop:ordinary-side}
  We have $\pi_2\bigl(M(\Wc_{5,4}(A))\bigr)\otimes\Fb_2\cong\Fb_2^{49}$.
\end{prop}

\begin{proof}
  As noted in Section~3.1, $M(\mathcal W_{5,4}(A))$ and $M_{6,4}$ have the same homotopy type.
  Applying Theorem~\ref{thm:BW} with $N=6$ and $k=4$, we obtain that $M_{6,4}$ is simply connected.
  Moreover, its integral cohomology is free abelian and is nonzero only in degrees $0$ and $2$, with
  \begin{equation*}
    \rk H^2(M_{6,4};\Zb)=49.
  \end{equation*}
  Thus $H^2(M_{6,4};\Zb)\cong\Zb^{49}$ and $H^3(M_{6,4};\Zb)=0$.
  Since $M_{6,4}$ is simply connected, we have $H_1(M_{6,4};\Zb)=0$.
  The universal coefficient theorem gives $\Hom(H_2(M_{6,4};\Zb),\Zb)\cong H^2(M_{6,4};\Zb)\cong\Zb^{49}$ and $\Ext(H_2(M_{6,4};\Zb),\Zb)=0$.
  Since $M_{6,4}$ has the homotopy type of a finite CW complex, $H_2(M_{6,4};\Zb)$ is finitely generated.
  It follows that $H_2(M_{6,4};\Zb)\cong\Zb^{49}$.
  By the Hurewicz theorem, $\pi_2(M_{6,4})\cong H_2(M_{6,4};\Zb)\cong\Zb^{49}$, hence $\pi_2\bigl(M(\Wc_{5,4}(A))\bigr)\cong\Zb^{49}$.
  Tensoring with $\Fb_2$ gives $\pi_2\bigl(M(\Wc_{5,4}(A))\bigr)\otimes\Fb_2\cong\Fb_2^{49}$.
\end{proof}

\section{The discrete side}\label{sec:discrete-side}
In this section, we prove the discrete part of Theorem~\ref{thm:main}.

\subsection{The graph $G_{5,4}$}
We first give an explicit combinatorial description of the graph $G_{5,4}=\Gamma^2(C(A_5))$.
Recall that the chambers of the Coxeter complex $C(A_5)$ are naturally indexed by the elements of $S_6$.
Hence the vertex set of $G_{5,4}$ is naturally identified with $S_6$, and $|V(G_{5,4})|=6!=720$.
For a permutation $w=(w_1,\ldots,w_6)\in S_6$, let $C_w$ denote the corresponding chamber.
Recall that the Coxeter complex $C(A_5)$ is the order complex of the Boolean lattice on $[6]$ with its minimal and maximal elements removed.
Under this identification, the chamber $C_w$ is the $4$-simplex corresponding to the maximal chain $\{w_1\}\subset\{w_1,w_2\}\subset\cdots\subset\{w_1,\ldots,w_5\}$.
Thus, writing $V_{w,i}=\{w_1,\ldots,w_i\}$ for $1\leq i\leq5$, the five vertices of $C_w$ are $V_{w,1},\ldots,V_{w,5}$.
Since $G_{5,4}=\Gamma^2(C(A_5))$, two chambers are adjacent if and only if they share at least three vertices.
It follows that two permutations $u,v\in S_6$ are adjacent in $G_{5,4}$ if and only if
\begin{equation*}
  \left|\left\{i\in\{1,\dots,5\}\ \middle|\ V_{u,i}=V_{v,i}\right\}\right|\geq3.
\end{equation*}

The connectivity of $G_{5,4}$ follows from the fact that the adjacent transpositions generate $S_6$.
The vanishing of its discrete fundamental group is a special case of \cite[Corollary~5.3]{BSW11}.
Hence, we have the following lemma.
\begin{lem}\label{lem:A1-vanishing}
  The graph $G_{5,4}$ is connected, and $A_1(G_{5,4})=0$.
\end{lem}

By Lemma~\ref{lem:A1-vanishing} and Theorem~\ref{thm:discrete-Hurewicz}, we have $A_2(G_{5,4})\cong\widetilde{DH}_2(G_{5,4},e)$.
Since $G_{5,4}$ is connected, reduced and unreduced discrete homology agree in degree two, and hence $A_2(G_{5,4})\cong DH_2(G_{5,4})$.
We will prove the vanishing $DH_2(G_{5,4};\Fb_2)=0$, from which the required vanishing $A_2(G_{5,4})\otimes\Fb_2=0$ will follow by the universal coefficient theorem.

\subsection{Support-rank filtration}
We introduce a filtration of the cubical chain complex by parabolic support.
Let $S=\{s_1,\ldots,s_5\}$ be the standard generating set of $S_6$.
For each $n\geq0$, let $u_0=(0,\ldots,0)\in I_1^{\Box n}$.

\begin{defi}
  A graph map $\sigma:I_1^{\Box n}\longrightarrow G_{5,4}$ is said to be \textbf{based} if $\sigma(u_0)=e$, where $e\in S_6=V(G_{5,4})$ denotes the identity.
\end{defi}

For a graph map $\sigma:I_1^{\Box n}\longrightarrow G_{5,4}$, define its \textbf{based translate} $\overline{\sigma}$ by $\overline{\sigma}(u)=\sigma(u_0)^{-1}\sigma(u)$.
Then $\overline{\sigma}$ is the unique based left translate of $\sigma$.

\begin{defi}
  The \textbf{support rank} of a cube $\sigma$, denoted by $\operatorname{srk}(\sigma)$, is the smallest integer $r$ for which there exists a standard parabolic subgroup $W_I$ of rank $r$ satisfying $\overline{\sigma}(I_1^{\Box n})\subseteq W_I$, where $\overline{\sigma}$ is the based translate of $\sigma$.
\end{defi}

The following lemma shows that the support rank is compatible with the left action of $S_6$ and with the cubical operators.

\begin{lem}\label{lem:support-rank}
  Let $\sigma:I_1^{\Box n}\to G_{5,4}$ be a graph map.
  \begin{enumerate}
    \item For every $g\in S_6$, we have $\operatorname{srk}(g\sigma)=\operatorname{srk}(\sigma)$.
    \item If $\tau$ is obtained from $\sigma$ by precomposition with a face, degeneracy, or connection map, then $\operatorname{srk}(\tau)\leq\operatorname{srk}(\sigma)$.
  \end{enumerate}
\end{lem}

\begin{proof}
  For (1), we have $\overline{g\sigma}(u)=(g\sigma(u_0))^{-1}g\sigma(u)=\sigma(u_0)^{-1}\sigma(u)=\overline{\sigma}(u)$.

  For (2), write $\tau=\sigma\theta$, where $\theta$ is a face, degeneracy, or connection map, and suppose that $\overline{\sigma}(I_1^{\Box n})\subseteq W_I$.
  Then $\sigma(I_1^{\Box n})\subseteq\sigma(u_0)W_I$, and hence $\tau(I_1^{\Box n'})\subseteq\sigma(u_0)W_I$, where $I_1^{\Box n'}$ is the domain of $\tau$.
  In particular, for every $u\in I_1^{\Box n'}$, both $\tau(u_0)$ and $\tau(u)$ belong to the same left coset $\sigma(u_0)W_I$.
  Therefore $\tau(u_0)^{-1}\tau(u)\in W_I$.
  It follows that $\overline{\tau}(I_1^{\Box n'})\subseteq W_I$, and hence $\operatorname{srk}(\tau)\leq\operatorname{srk}(\sigma)$.
\end{proof}

In particular, cubes of support rank at most $r$ form a cubical subset of $N_1G_{5,4}$.
For $0\leq r\leq5$, let $F_rC_q^D(G_{5,4};\Fb_2)$ be the subspace generated by the non-degenerate $q$-cubes of support rank at most $r$.
By Lemma~\ref{lem:support-rank}, these subspaces form a subcomplex $F_rC_*^D(G_{5,4};\Fb_2)\subseteq C_*^D(G_{5,4};\Fb_2)$.
For brevity, write $F_r=F_rC_*^D(G_{5,4};\Fb_2)$.
We obtain a filtration
\begin{equation*}
  F_0\subseteq F_1\subseteq\cdots\subseteq F_5=C_*^D(G_{5,4};\Fb_2).
\end{equation*}
Since support rank is invariant under left translation, every cube has a unique based representative in its left $S_6$-orbit.
We use based cubes to describe the support filtration and, computationally, to enumerate orbit representatives.
The chain complex itself is not quotiented by the left $S_6$-action.

We say that a cube $\sigma:I_1^{\Box n}\to G_{5,4}$ has \textbf{exact support} $J\subseteq S$ if $\overline{\sigma}(I_1^{\Box n})\subseteq W_J$ and $\overline{\sigma}(I_1^{\Box n})\nsubseteq W_I$ for every proper subset $I\subsetneq J$.
For $r\geq1$, a subset $J\subseteq S$ with $|J|=r$, and a left coset $gW_J$, let $C_q^{g,J}$ denote the subspace of $C_q(F_r,F_{r-1})$ generated by the classes of non-degenerate $q$-cubes $\sigma$ having exact support $J$ and satisfying $\sigma(u_0)W_J=gW_J$.
Equivalently, such a cube satisfies $\sigma(I_1^{\Box q})\subseteq gW_J$.
We write $C_*^{g,J}$ for the resulting graded subspace.

\begin{lem}\label{lem:support-decomposition}
  For each $1\leq r\leq5$, there is a direct sum decomposition of chain complexes
  \begin{equation}\label{eq:support-decomposition}
    C_*(F_r,F_{r-1})=\bigoplus_{\substack{J\subseteq S\\ |J|=r}}\ \bigoplus_{gW_J\in S_6/W_J}C_*^{g,J}.
  \end{equation}
  Moreover, left translation induces an isomorphism of chain complexes $C_*^{g,J}\cong C_*^{e,J}$.
\end{lem}

\begin{proof}
  Since standard parabolic subgroups satisfy $W_I\cap W_J=W_{I\cap J}$, every cube has a unique exact support.
  Indeed, the intersection of all standard parabolic subgroups containing the image of its based translate is again a standard parabolic subgroup.

  We next show that each $C_*^{g,J}$ is a subcomplex.
  Suppose that $\sigma$ has exact support $J$ with $|J|=r$ and $\sigma(I_1^{\Box q})\subseteq gW_J$.
  Every face $\tau$ of $\sigma$ is also contained in $gW_J$.
  Since $\tau(u_0)\in gW_J$, we have $\overline{\tau}(I_1^{\Box(q-1)})\subseteq W_J$.
  Thus the exact support of $\tau$ is a subset of $J$.
  If $\tau$ has support rank less than $r$, it vanishes in $C_*(F_r,F_{r-1})$.
  Otherwise, its exact support is $J$, and its initial vertex belongs to the same coset $gW_J$.
  Hence the relative boundary preserves $C_*^{g,J}$.
  Every non-degenerate cube of exact support rank $r$ belongs to exactly one such summand, determined by its exact support $J$ and the coset
  $\sigma(u_0)W_J$.
  Since these classes form a basis of $C_*(F_r,F_{r-1})$, the asserted direct sum decomposition follows.

  Finally, choosing a representative $g$ of the left coset $gW_J$, left translation by $g^{-1}$ preserves graph maps, non-degenerate cubes, and their boundaries, and identifies $C_*^{g,J}$ with $C_*^{e,J}$.
\end{proof}

We denote the boundary maps of $C_*^{g,J}$ by $\partial_q^{g,J}:C_q^{g,J}\to C_{q-1}^{g,J}$.
By Lemma~\ref{lem:support-decomposition}, the local relative homology depends only on the standard support $J$, and the corresponding local complexes occur with multiplicity $[S_6:W_J]$.

We use the decomposition~\eqref{eq:support-decomposition} to compute the relative homology groups and the connecting homomorphisms over $\Fb_2$.

\subsection{Computation of the second discrete homology}\label{subsec:computation}
We begin with the homology of $F_1$ and the relative homology of the pair $(F_2,F_1)$.

\begin{lem}\label{lem:rank2-computer}
  The following hold:
  \begin{equation*}
    H_1(F_1;\Fb_2)\cong\Fb_2^{1081},\qquad H_2(F_1;\Fb_2)=0,\qquad H_2(F_2,F_1;\Fb_2)\cong\Fb_2^{1560}.
  \end{equation*}
  Moreover, the connecting homomorphism $\delta_2:H_2(F_2,F_1;\Fb_2)\longrightarrow H_1(F_1;\Fb_2)$ has rank $1081$.
\end{lem}
\begin{proof}
  The computations in Appendix~\ref{app:through-f2} give $\dim H_1(F_1;\Fb_2)=1081$ and $H_2(F_1;\Fb_2)=0$.
  For each exact support $J$ of rank two, the local second homology is one-dimensional (see Table~\ref{tab:app-rank12}).
  By the decomposition~\eqref{eq:support-decomposition}, these local computations give $H_2(F_2,F_1;\Fb_2)\cong\Fb_2^{1560}$.
  Finally, the computation of the connecting homomorphism gives $\rk\delta_2=1081$ (see Table~\ref{tab:app-delta2}).
\end{proof}

\begin{prop}\label{prop:H2-F2}
  We have
  \begin{equation*}
    H_2(F_2;\Fb_2)\cong\Fb_2^{479}.
  \end{equation*}
\end{prop}
\begin{proof}
  By the long exact sequence of the pair $(F_2,F_1)$ and Lemma~\ref{lem:rank2-computer}, we have $H_2(F_2;\Fb_2)\cong\ker\delta_2$, which has dimension $1560-1081=479$.
\end{proof}

We next study the pair $(F_3,F_2)$.
\begin{lem}\label{lem:rank3-computer}
  We have
  \begin{equation*}
    H_2(F_3,F_2;\Fb_2)=0.
  \end{equation*}
  Moreover, the connecting homomorphism $\delta_3:H_3(F_3,F_2;\Fb_2)\longrightarrow H_2(F_2;\Fb_2)$ has rank $479$.
\end{lem}
\begin{proof}
  By the decomposition~\eqref{eq:support-decomposition}, it suffices to consider the local relative complexes for the supports $J$ of rank three.
  There are three supports of Coxeter type $A_3$, six of type $A_2\times A_1$, and one of type $A_1\times A_1\times A_1$.
  The computation in Appendix~\ref{app:rank-three} shows that $\partial_3^{g,J}:C_3^{g,J}\longrightarrow C_2^{g,J}$ is surjective for every such support (see Table~\ref{tab:app-rank3}).
  Consequently, $H_2(F_3,F_2;\Fb_2)=0$.
  The same computation produces $479$ linearly independent classes in the image of $\delta_3$ (see Table~\ref{tab:app-delta3}).
  Since Proposition~\ref{prop:H2-F2} gives $\dim H_2(F_2;\Fb_2)=479$, it follows that $\rk\delta_3=479$.
\end{proof}

We finally study the pair $(F_4,F_3)$.
\begin{lem}\label{lem:rank4-computer}
  We have
  \begin{equation*}
    H_2(F_4,F_3;\Fb_2)=0.
  \end{equation*}
  Moreover, every based non-degenerate $2$-cube has support rank at most four.
\end{lem}
\begin{proof}
  By the decomposition~\eqref{eq:support-decomposition}, it suffices to consider the local relative complexes for the five standard supports $J$ of rank four. These have Coxeter types $A_4$, $A_3\times A_1$, and $A_2\times A_2$.
  The computation in Appendix~\ref{app:rank-four} shows that $\partial_3^{g,J}:C_3^{g,J}\longrightarrow C_2^{g,J}$ is surjective for every such support (see Table~\ref{tab:app-rank4}).
  Consequently, $H_2(F_4,F_3;\Fb_2)=0$.
  Finally, the complete enumeration in Appendix~\ref{app:rank-four} shows that every based non-degenerate $2$-cube has support rank at most four (see Table~\ref{tab:app-support-distribution}).
\end{proof}

\begin{prop}\label{prop:discrete-homology}
  We have
  \begin{equation*}
    DH_2(G_{5,4};\Fb_2)=0.
  \end{equation*}
\end{prop}

\begin{proof}
  By Proposition~\ref{prop:H2-F2} and Lemma~\ref{lem:rank3-computer}, the connecting homomorphism
  \begin{equation*}
    \delta_3:H_3(F_3,F_2;\Fb_2)\longrightarrow H_2(F_2;\Fb_2)
  \end{equation*}
  is surjective.
  Since $H_2(F_3,F_2;\Fb_2)=0$, the long exact sequence of the pair $(F_3,F_2)$ gives $H_2(F_3;\Fb_2)=0$.
  Similarly, Lemma~\ref{lem:rank4-computer} gives
  $H_2(F_4,F_3;\Fb_2)=0$.
  The long exact sequence of the pair $(F_4,F_3)$ then implies $H_2(F_4;\Fb_2)=0$.
  Finally, by Lemmas~\ref{lem:support-rank} and~\ref{lem:rank4-computer}, every non-degenerate $2$-cube has support rank at most four, so
  \begin{equation*}
    (F_4)_2=C_2^D(G_{5,4};\Fb_2).
  \end{equation*}
  Hence every $2$-cycle in the full chain complex is a $2$-cycle in $F_4$.
  Since $H_2(F_4;\Fb_2)=0$, it is already a boundary in $F_4$, and therefore also in the full chain complex.
  Thus $DH_2(G_{5,4};\Fb_2)=0$.
\end{proof}

\begin{cor}\label{cor:A2-vanishing}
  We have
  \begin{equation*}
    A_2(G_{5,4})\otimes\Fb_2=0.
  \end{equation*}
\end{cor}
\begin{proof}
  By Lemma~\ref{lem:A1-vanishing} and Theorem~\ref{thm:discrete-Hurewicz}, we have $A_2(G_{5,4})\cong DH_2(G_{5,4})$.
  The universal coefficient theorem for homology gives an injection
  \begin{equation*}
    DH_2(G_{5,4})\otimes\Fb_2\longrightarrow DH_2(G_{5,4};\Fb_2).
  \end{equation*}
  Since the target vanishes by Proposition~\ref{prop:discrete-homology}, we obtain $A_2(G_{5,4})\otimes\Fb_2=0$.
\end{proof}

\appendix
\section{Computational verification}\label{app:computation}
This appendix records the computer-assisted calculations used in Section~\ref{sec:discrete-side}. All computations are performed over $\Fb_2$ in the normalized cubical chain complex of $N_1G_{5,4}$, using the support-rank filtration introduced in Section~\ref{sec:discrete-side} and the decomposition~\eqref{eq:support-decomposition}.

\subsection{Computational setup}
A $q$-cube $\sigma:I_1^{\Box q}\longrightarrow G_{5,4}$ is represented by its values on the $2^q$ vertices of $I_1^{\Box q}$. For computational purposes, we enumerate based cubes as representatives of left $S_6$-orbits, retaining all their left translates as distinct generators of the chain complex.

For each retained cube, its exact parabolic support is determined and its boundary over $\Fb_2$ is computed by
\begin{equation*}
  \partial\sigma=\sum_{i=1}^q\sum_{\varepsilon=0}^1\sigma\partial_{i,\varepsilon}.
\end{equation*}
Faces that are degenerate, including those arising from connection maps, are zero in the normalized chain complex. In the relative complex $C_*(F_r,F_{r-1})$, faces of support rank less than $r$ are additionally discarded.
Thus the relative boundary computations are carried out separately for each summand in the decomposition~\eqref{eq:support-decomposition}.

\textbf{Computational environment.}
The computations were performed using SageMath~\cite{The26} (version 10.8) with Python 3.13.7 on an Apple Silicon-based Mac running macOS.

\subsection{Computations through the second filtration stage}\label{app:through-f2}
The direct computation of the normalized chain groups of $F_1$ gives $\dim C_1(F_1)=3600$, $\dim C_2(F_1)=10800$, $\rk\bigl(\partial_1:C_1(F_1)\to C_0(F_1)\bigr)=719$, and $\rk\bigl(\partial_2:C_2(F_1)\to C_1(F_1)\bigr)=1800$.
Consequently,
\begin{equation*}
  \dim H_1(F_1;\Fb_2)=3600-719-1800=1081.
\end{equation*}

For each standard support $J$ of rank one or two, we exhaustively enumerate the non-degenerate cubes up to dimension three in the local exact-support complex $C_*^{e,J}$ and directly compute the ranks of its boundary maps.

For every exact support $J$ of rank one, we obtain $\dim C_2^{g,J}=6$, $\rk\partial_2^{g,J}=1$, and $\rk\partial_3^{g,J}=5$.
Thus $\dim H_2(C_*^{g,J};\Fb_2)=0$.
By the decomposition~\eqref{eq:support-decomposition}, $H_2(F_1,F_0;\Fb_2)=0$.
Since $F_0$ is concentrated in degree zero, the long exact sequence of $(F_1,F_0)$ gives
\begin{equation*}
  H_2(F_1;\Fb_2)=0.
\end{equation*}

For supports of rank two, there are four standard supports of Coxeter type $A_2$ and six of type $A_1\times A_1$. The local chain dimensions and boundary ranks are recorded in Table~\ref{tab:app-rank12}.

\begin{table}[htbp]
  \centering
  \begin{tabular}{c|r|r|r|r|r}
    \text{type}
                    & $\#\{J\}$ & $[S_6:W_J]$ & $\dim C_2^{g,J}$ & $\rk\partial_2^{g,J}$ & $\rk\partial_3^{g,J}$ \\ \hline
    $A_2$           & 4         & 120         & 1134             & 18                    & 1115                  \\
    $A_1\times A_1$ & 6         & 180         & 180              & 4                     & 175
  \end{tabular}\caption{Local boundary computations in support rank two.}\label{tab:app-rank12}
\end{table}

For each support of rank two, Table~\ref{tab:app-rank12} gives $\dim H_2(C_*^{g,J};\Fb_2)=\dim C_2^{g,J}-\rk\partial_2^{g,J}-\rk\partial_3^{g,J}=1$.
The decomposition~\eqref{eq:support-decomposition}  therefore yields
\begin{equation*}
  \dim H_2(F_2,F_1;\Fb_2)=4\cdot120+6\cdot180=1560.
\end{equation*}

We now compute $\delta_2:H_2(F_2,F_1;\Fb_2)\longrightarrow H_1(F_1;\Fb_2)$ using explicit relative $2$-cycles.

For each standard support $J$ with $|J|=2$, the local computation retains a relative $2$-cycle representing the non-zero class in $H_2(C_*^{e,J};\Fb_2)$.
We translate this cycle by representatives of the left cosets $gW_J$. For each translated cycle, its full normalized boundary is computed in $C_1(F_1)$ and reduced modulo $\ima\bigl(\partial_2:C_2(F_1)\to C_1(F_1)\bigr)$.
The resulting $1560$ columns are processed by Gaussian elimination over $\Fb_2$.

The rank growth obtained by processing the supports in lexicographic order is recorded in Table~\ref{tab:app-delta2}.
Here and below, we abbreviate $\{s_{i_1},\ldots,s_{i_r}\}$ by $\{i_1,\ldots,i_r\}$.

\begin{table}[htbp]
  \centering
  \begin{tabular}{c|r|r|r}
    $J$       & $[S_6:W_J]$ & \text{new rank} & \text{cumulative rank} \\ \hline
    $\{1,2\}$ & 120         & 120             & 120                    \\
    $\{1,3\}$ & 180         & 180             & 300                    \\
    $\{1,4\}$ & 180         & 180             & 480                    \\
    $\{1,5\}$ & 180         & 180             & 660                    \\
    $\{2,3\}$ & 120         & 90              & 750                    \\
    $\{2,4\}$ & 180         & 120             & 870                    \\
    $\{2,5\}$ & 180         & 120             & 990                    \\
    $\{3,4\}$ & 120         & 36              & 1026                   \\
    $\{3,5\}$ & 180         & 45              & 1071                   \\
    $\{4,5\}$ & 120         & 10              & 1081
  \end{tabular}
  \caption{Rank computation for the connecting homomorphism $\delta_2$.}
  \label{tab:app-delta2}
\end{table}

The individual rank increments depend on the processing order, whereas the final cumulative rank does not.
The computation gives
\begin{equation*}
  \rk\delta_2=1081.
\end{equation*}

\subsection{Computations at support rank three}\label{app:rank-three}
For each standard support $J$ of rank three, we enumerate all non-degenerate $2$-cubes generating $C_2^{e,J}$.
To prove that $\partial_3^{e,J}:C_3^{e,J}\to C_2^{e,J}$ is surjective, it is unnecessary to enumerate all non-degenerate $3$-cubes.

Instead, for each enumerated based $2$-cube, we first construct non-degenerate $3$-cubes whose bottom face is the given square and whose top face is constant, which we call cones.
If their relative boundaries do not span $C_2^{e,J}$, we additionally construct $3$-cubes with non-constant top faces in $W_J$, which we call prisms.
In both cases, the top vertices are required to satisfy the graph-map adjacency conditions.
The relative boundaries of their left $W_J$-translates are reduced by Gaussian elimination over $\Fb_2$ until they span $C_2^{e,J}$.
Each retained independent boundary is recomputed directly from its source $3$-cube.

The resulting dimensions and boundary ranks are recorded in Table~\ref{tab:app-rank3}.

\begin{table}[htbp]
  \centering
  \begin{tabular}{c|r|r}
    \text{Coxeter type of }$W_J$ & $\dim C_2^{g,J}$ & $\rk\partial_3^{g,J}$ \\ \hline
    $A_3$                        & 7728             & 7728                  \\
    $A_2\times A_1$              & 2856             & 2856                  \\
    $A_1\times A_1\times A_1$    & 1056             & 1056
  \end{tabular}
  \caption{Surjectivity computations in support rank three.}
  \label{tab:app-rank3}
\end{table}

In every case, $\rk\partial_3^{e,J}=\dim C_2^{e,J}$.
By left translation and the decomposition~\eqref{eq:support-decomposition}, this proves
\begin{equation*}
  H_2(F_3,F_2;\Fb_2)=0.
\end{equation*}

We next compute $\delta_3:H_3(F_3,F_2;\Fb_2)\longrightarrow H_2(F_2;\Fb_2)$.
For a non-degenerate $3$-cube $\sigma$ of exact support rank three, decompose its full boundary as $\partial\sigma=a_\sigma+b_\sigma$, where $a_\sigma$ is its exact rank-three component and $b_\sigma\in C_2(F_2)$.
A linear relation $\sum_\sigma c_\sigma a_\sigma=0$ determines a relative $3$-cycle whose full boundary is $\sum_\sigma c_\sigma b_\sigma\in C_2(F_2)$.
Since $\partial^2=0$, this boundary is a $2$-cycle representing the corresponding connecting image.

To organize the rank-two boundary components, we use the auxiliary quotient
\begin{equation*}
  Q=\bigoplus_{\substack{J\subseteq S\\ |J|=2}} \ \bigoplus_{gW_J\in S_6/W_J}C_2^{g,J}/\ima\partial_3^{g,J}.
\end{equation*}
The implementation records in $Q$ the rank-two exact-support component of each lower-support boundary, discarding components of support rank at most one. By Table~\ref{tab:app-rank12},
\begin{equation*}
  \dim Q=4\cdot120\cdot(1134-1115)+6\cdot180\cdot(180-175)=14520.
\end{equation*}

The space $Q$ is an auxiliary quotient, not a homology group.
For each support $J$ of rank two, the inclusion $\ker\partial_2^{g,J}\subseteq C_2^{g,J}$ induces an injection $H_2(C_*^{g,J};\Fb_2)\lhook\joinrel\longrightarrow C_2^{g,J}/\ima\partial_3^{g,J}$.
Taking direct sums gives
\begin{equation*}
  H_2(F_2,F_1;\Fb_2)\lhook\joinrel\longrightarrow Q.
\end{equation*}
Each local second homology group is one-dimensional, and its non-zero class is retained as an explicit generator.
These generators are used to recognize and decode relative homology classes in $Q$.

The connecting images are computed as follows.

\begin{enumerate}
  \item \textbf{Relative cycles.}
        For each support $J$ of rank three, process the relative boundaries of non-degenerate $3$-cubes by Gaussian elimination, retaining linear dependencies together with their source $3$-chains.

  \item \textbf{Elimination after full rank.}
        Once the relative boundaries span $C_2^{e,J}$, construct $3$-chains $u_k$ satisfying $\partial_3^{e,J}u_k=e_k$ for each basis element $e_k$ of $C_2^{e,J}$.
        For a subsequent $3$-cube $\sigma$ with $a_\sigma=\sum_k a_ke_k$, the chain $z_\sigma=\sigma+\sum_k a_ku_k$ is a relative $3$-cycle.
        Its rank-two boundary component in $Q$ is computed using the stored boundary components of the $u_k$, without further Gaussian elimination.

  \item \textbf{Connecting images.}
        Decode the rank-two component of each resulting boundary in $Q$ as a class in $H_2(F_2,F_1;\Fb_2)$ and verify that it belongs to $\ker\delta_2\cong H_2(F_2;\Fb_2)$.

  \item \textbf{Independent images.}
        Compute left translates by representatives of $S_6/W_J$ and retain linearly independent connecting images.
        The corresponding source $3$-cycles and their full boundaries are verified by direct recomputation.
\end{enumerate}

The computation retains $10$ original relative $3$-cycles whose left translates yield $479$ linearly independent connecting images.
The support-by-support rank growth is recorded in Table~\ref{tab:app-delta3}.

\begin{table}[htbp]
  \centering
  \begin{tabular}{c|r|r|r}
    $J$         & $\dim C_2^{g,J}$ & \text{new rank} & \text{cumulative rank} \\ \hline
    $\{1,2,3\}$ & 7728             & 30              & 30                     \\
    $\{1,2,4\}$ & 2856             & 60              & 90                     \\
    $\{1,2,5\}$ & 2856             & 60              & 150                    \\
    $\{1,3,4\}$ & 2856             & 60              & 210                    \\
    $\{1,3,5\}$ & 1056             & 90              & 300                    \\
    $\{1,4,5\}$ & 2856             & 60              & 360                    \\
    $\{2,3,4\}$ & 7728             & 24              & 384                    \\
    $\{2,3,5\}$ & 2856             & 45              & 429                    \\
    $\{2,4,5\}$ & 2856             & 40              & 469                    \\
    $\{3,4,5\}$ & 7728             & 10              & 479
  \end{tabular}
  \caption{Rank computation for the connecting homomorphism $\delta_3$.}
  \label{tab:app-delta3}
\end{table}

As in Table~\ref{tab:app-delta2}, the individual rank increments depend on the processing order.
The final cumulative rank is $479$, giving
\begin{equation*}
  \rk\delta_3=479.
\end{equation*}

\subsection{Rank-four computation and verification}\label{app:rank-four}
For each standard support $J$ of rank four, we apply the same cone-and-prism construction and Gaussian elimination procedure used in support rank three.
The resulting boundary ranks are recorded in Table~\ref{tab:app-rank4}.

\begin{table}[htbp]
  \centering
  \begin{tabular}{c|r|r}
    \text{Coxeter type of }$W_J$ & $\dim C_2^{g,J}$ & $\rk\partial_3^{g,J}$ \\ \hline
    $A_4$                        & 12000            & 12000                 \\
    $A_3\times A_1$              & 4032             & 4032                  \\
    $A_2\times A_2$              & 4248             & 4248
  \end{tabular}
  \caption{Surjectivity computations in support rank four.}
  \label{tab:app-rank4}
\end{table}

For every support $J$ of rank four, $\rk\partial_3^{e,J}=\dim C_2^{e,J}$.
By left translation and the decomposition~\eqref{eq:support-decomposition}, this proves
\begin{equation*}
  H_2(F_4,F_3;\Fb_2)=0.
\end{equation*}

The complete enumeration of based non-degenerate $2$-cubes gives the support-rank distribution recorded in Table~\ref{tab:app-support-distribution}.

\begin{table}[htbp]
  \centering
  \begin{tabular}{c|rrrr|r}
    \text{support rank}
     & 1  & 2    & 3    & 4   & \text{total} \\ \hline
    \text{number of based non-degenerate $2$-cubes}
     & 15 & 1026 & 2526 & 486 & 4053
  \end{tabular}
  \caption{Support-rank distribution of based non-degenerate $2$-cubes.}
  \label{tab:app-support-distribution}
\end{table}

\noindent\textbf{Consistency checks.}
For support ranks three and four, each retained independent relative boundary is recomputed directly from its source $3$-cube, and the resulting boundaries are verified to span $C_2^{e,J}$.
For $\delta_3$, the retained source $3$-cycles and their left translates are reconstructed, and their full normalized boundaries are checked against the computed images in $Q$.
Finally, the $479$ connecting images are verified to be linearly independent and to represent classes in $\ker\delta_2$.

\medskip
\noindent\textbf{Source code and reproducibility.}
The SageMath source code, together with the complete execution log, is publicly available at \url{https://github.com/so-ipa/BSW_conj_computations}.

\section*{Acknowledgments}
The author would like to thank Luigi Caputi for helpful discussions on this project, as well as for carefully reading an earlier version of the manuscript and providing valuable comments and suggestions.
The author was supported by JSPS KAKENHI Grant Number JP24K16926 and JSPS Bilateral Program Grant Number JPJSBP120256504.

\section*{Use of AI}
Generative AI tools were used for language editing and to assist in the development of the computational code.
The author reviewed the mathematical arguments and verified the computational results using the procedures described in Appendix~\ref{app:computation} and takes full responsibility for the contents of the manuscript.

\bibliographystyle{amsalpha}
\bibliography{kpara}

\end{document}